\documentclass[12pt]{amsart}
\usepackage{amscd,verbatim}
\usepackage{amssymb}
\usepackage[all]{xy}

\newcommand{\F}{\mathbb{F}}
\newcommand{\G}{\mathbb{G}}

\renewcommand{\P}{\mathbb{P}}
\newcommand{\Q}{\mathbb{Q}}
\newcommand{\Z}{\mathbb{Z}}
\newcommand{\R}{\mathbb{R}}
\newcommand{\C}{\mathbb{C}}
\newcommand{\E}{\mathbb{E}}
\newcommand{\T}{\mathbb{T}}
\newcommand{\W}{\mathbb{W}}

\newcommand{\sC}{\mathcal{C}}
\newcommand{\sD}{\mathcal{D}}
\newcommand{\sE}{\mathcal{E}}

\newcommand{\sM}{\mathcal{M}}

\newcommand{\End}{\operatorname{End}}

\newcommand{\IM}{\operatorname{Im}}
\newcommand{\Coker}{\operatorname{Coker}}

\newcommand{\Tor}{{\operatorname{Tor}}}
\newcommand{\Div}{\operatorname{Div}}
\newcommand{\ord}{\operatorname{ord}}

\newcommand{\GL}{\operatorname{GL}}
\newcommand{\SL}{\operatorname{SL}}
\newcommand{\aug}{\operatorname{aug}}

\renewcommand{\deg}{\operatorname{deg}}

\renewcommand{\epsilon}{\varepsilon}
\renewcommand{\div}{\operatorname{div}}

\newcommand{\tJ}{{\widetilde{J}}}
\newcommand{\tM}{{\widetilde{M}}}
\newcommand{\tT}{{\widetilde{T}}}
\newcommand{\tC}{{\widetilde{\sC}}}

\swapnumbers
\newtheorem{lemma}{Lemma}[subsection]

\newtheorem{theorem}[lemma]{Theorem}
\newtheorem{prop}[lemma]{Proposition}
\newtheorem{proposition}[lemma]{Proposition}
\newtheorem{cor}[lemma]{Corollary}
\newtheorem{corollary}[lemma]{Corollary}
\theoremstyle{definition}

\theoremstyle{remark}

\newtheorem{remark}[lemma]{Remark}

\numberwithin{equation}{section}

\begin{document}

\title[Generalized modular and Drinfeld modular Jacobians]
{Rational torsion of generalized Jacobians\\
of modular and Drinfeld modular curves}

\author{Fu-Tsun Wei \and Takao Yamazaki}
\date{\today}

\address{Department of Mathematics, National Tsing Hua University, Taiwan}
\email{ftwei@math.nthu.edu.tw}

\address{Mathematical Institute, Tohoku University,
Aoba, Sendai 980-8578, Japan}
\email{ytakao@math.tohoku.ac.jp}

\begin{abstract}
We consider the generalized Jacobian $\widetilde{J}$
of the modular curve $X_0(N)$ of level $N$
with respect to a reduced divisor consisting of all cusps.
Supposing $N$ is square free,
we explicitly determine the structure of the
$\mathbb{Q}$-rational torsion points on $\widetilde{J}$
up to $6$-primary torsion. 
%The result turns out to be
%very different from the case of prime power level
%previously studied by Yang and the second author.
The result depicts a fuller picture than \cite{YY}
where the case of prime power level was studied.
We also obtain an analogous result for Drinfeld modular curves.
Our proof relies on  similar results for classical Jacobians
due to Ohta, Papikian and the first author.
We also discuss the Hecke action on  $\widetilde{J}$
and its Eisenstein property.
\end{abstract}

\thanks{The first author is supported by MOST Grant (105-2115-M-007-018-MY2 \& 107-2628-M-007-004-MY4). \indent The second author is supported by JSPS KAKENHI Grant (15K04773).}
\maketitle

\section{Introduction}

\subsection{Background and overview}
Let $J$ be the Jacobian variety
of the modular curve $X:=X_0(p)$ over $\Q$ of prime level $p$.
In a celebrated paper \cite{Mazur}, Mazur proved that 
the group of $\Q$-rational torsion points $J(\Q)_\Tor$ on $J$
is a cyclic group of order $(p-1)/(p-1, 12)$.
(Here and henceforth 
we denote by $(a, b)$ the greatest common divisor of $a$ and $b$.)
This result has been generalized 
to prime power level
by Lorenzini \cite{Lorenzini} and Ling \cite{Ling},
as well as 
to square free level
by Ohta \cite{Ohta}.
The latter result will be recalled in Theorem \ref{thm:ohta-pw} below.

A research toward a different direction was started in \cite{YY}.
Let $C$ be the set of all cusps on $X$, 
which we regard as a reduced effective divisor on $X$.
We consider the generalized Jacobian $\tJ$ of $X$
with respect to modulus $C$,
in the sense of Rosenlicht-Serre \cite{Serre}.
Since $C$ consists of two $\Q$-rational points in this case,
$\tJ$ is an extension of $J$ by $\G_m$
so that we have an exact sequence
(cf.\ \eqref{eq:delta})
\begin{equation}\label{eq:exact1} 
1 \to \{ \pm 1 \} \overset{i}{\to} \tJ(\Q)_\Tor \to J(\Q)_\Tor. 
\end{equation} 
In \cite{YY} it is shown that $i$ is an isomorphism.
It is also proved in loc.\ cit.\ that 
a similar bijectivity result holds
when the level is a power of a prime $\ge 5$,
conditional to a folklore conjecture
\lq\lq the cupsidal divisor classes cover all torsion rational points\rq\rq
(cf. \cite[Conjecture 2]{Ogg}).
%%conditional to the analogue of Ogg's conjecture:
%\lq\lq the cupsidal divisor classes cover all torsion rational points\rq\rq.
%{(\bf I think this is called Ogg's conjecture)}
On the contrary, it is observed that $i$ is far from being an isomorphism
when the level is a product of two different primes of certain type
\cite[Proposition 1.3.2]{YY}.

The purpose of the present article is twofold.
One is to clarify what happens in the case of square free level.
The other is to develop a parallel story 
for the rank two Drinfeld modular curves, again for square free level.
Our main result, explained in the next subsection,
pinpoints where $i$ fails to be an isomorphism (see Remark \ref{rem 1.2.3}).
Our proof relies on the study of classical Jacobians
due to Ohta \cite{Ohta} and to Papikian-Wei \cite{PW}.
%Unfortunately, we cannot deduce a result for Drinfeld modular curves
%of prime power level, 
%because no such result is known for the classical Jacobian. 
We also discuss the action of Hecke operators on 
$\widetilde{J}(F)_\Tor$ and study its Eisenstein property,
see \S \ref{sec:intro-eis}.

\subsection{Main results}\label{sec:intro-sqfree}
We work in either of the following setting:
\begin{itemize}
\item[(NF)]
Set $F=\Q$ and $A=\Z$. %, and $A_+=\Z_{>0}$. %, and $a=(p-1)/(p-1, 12)$.
Let $p_1, \dots, p_s$ be distinct primes, and put $N=p_1 \dots p_s \in A$.
Let $X:=X_0(N)$ be the modular curve with respect to 
$$\Gamma_0(N) = 
\left\{ \begin{pmatrix} a & b \\ c & d \end{pmatrix}
%\{ (\begin{smallmatrix} a & b \\ c & d \end{smallmatrix})
\in \SL_2(A)
\mid c \equiv 0 \bmod N \right\}.$$
\item[(FF)]
Let $\F_q$ be a finite field with $q$ elements. %, and let $a=q^2-1$.
Set $F=\F_q(t)$ and $A=\F_q[t]$.
%, and let $A_+ \subset A$ be the set of all monic polynomials.
Let $p_1, \dots, p_s \in A$ be 
distinct irreducible monic polynomials, and put $N=p_1 \dots p_s \in A$.
Let $X:=X_0(N)$ be the rank two Drinfeld modular curve 
with respect to 
$$\Gamma_0(N) = 
\left\{ \begin{pmatrix} a & b \\ c & d \end{pmatrix}
%\{ (\begin{smallmatrix} a & b \\ c & d \end{smallmatrix})
\in \GL_2(A)
\mid c \equiv 0 \bmod N \right\}.$$
\end{itemize}

We regard $X$ as 
a smooth projective absolutely integral curve over $F$.
Let $C \subset X$ be the reduced closed subscheme consisting of all cusps on $X$.
All points of $C$ are $F$-rational and $|C|=2^s$.
As before, we denote by $J$ and $\tJ$
the Jacobian variety of $X$
and the generalized Jacobian of $X$ with modulus $C$,
respectively.
Then  $\tJ$ is an extension of $J$ by 
the product of $2^s-1$ copies of $\G_m$
so that we have an exact sequence
(cf. \eqref{eq:delta})
\begin{equation}\label{eq:exact2} 
1 \to \mu_{F}^{\oplus (2^s-1)} \overset{i}{\to} \tJ(F)_\Tor \to J(F)_\Tor,
\end{equation} 
where $\mu_F$ consists of the roots of unity in $F$.
Our result can be stated easily when $s=1$
(that is, $N=p_1 \in A$ is a prime element).

\begin{theorem}\label{thm:main1}
If $s=1$, 
then the map $i$ is an isomorphism.
\end{theorem}

As explained above, the case (NF) has been proved in \cite{YY}.
Note also that in the case (FF) 
P\'al showed
$J(F)_\Tor$ is a cyclic group
of order $q^d/(q^2-1, q^d-1)$ with $d=\deg(p)$ in  \cite{Pal}.

For general $s$, we need to introduce notations.
%For a map $e : \{1, \dots, s \} \to \{ \pm 1\}$, 
%we define
%\[ N(e)=\prod_{i=1}^s(|p_i|+e(i)) \in \Z_{>0}, \]
For $j \in \Z_{\ge 0}$ %$j=1, 2$
we define 
\begin{align}
\label{eq:def-Ej}
&\E_j := \left\{ e : \{ 1, \dots, s \} \to \{ \pm 1 \} 
~\middle|~
|e^{-1}(-1)| \ge j \right\},
\\
\label{eq:def-de)}
&\sM_j:= \bigoplus_{e \in \E_j} 
\left( \Z/d(e) \Z \right),
\quad d(e):=\prod_{i=1}^s(|p_i| + e(i)).
\end{align}
Here
$|e^{-1}(-1)|$ denotes the cardinality of %the set
$\{ c \in \{ 1, \dots, s \} \mid e(c)=-1 \}$,
and
in the case (FF) we write $|p_i|:=q^{\deg(p_i)}$.
We set 
\[
a=\begin{cases}
6 & \text{in the case (NF)},\\
q(q^2-1) & \text{in the case (FF)}.
\end{cases}
\]
The group of $F$-rational torsion points on $J$
is studied by Ohta and Papikian-Wei
(see Theorem \ref{thm:ohta-pw2} below for more detailed results).

\begin{theorem}[{Ohta \cite[Theorem (3.6.2)]{Ohta}, Papikian-Wei \cite[Theorem 4.3]{PW}}]
\label{thm:ohta-pw}
There is an isomorphism
\[ J(F)_\Tor \otimes \Z\left[ \frac{1}{a} \right]
\cong \sM_1 \otimes \Z\left[ \frac{1}{a} \right].
\]
\end{theorem} 

Our main result for general $s$ is the following:

\begin{theorem}\label{thm:main2}
There is an isomorphism
\[ \tJ(F)_\Tor \otimes \Z\left[ \frac{1}{a} \right]
\cong \sM_2 \otimes \Z\left[ \frac{1}{a} \right].
\]
\end{theorem}

\begin{remark}\label{rem 1.2.3}
Both $J$ and $\tJ$ admit an action of $\W:=(\Z/2\Z)^s$ 
through the Atkin-Lehner involutions.
We identify $\E_0$ with the character group of $\W$.
Theorems \ref{thm:ohta-pw} and \ref{thm:main2} actually
describe the decomposition of 
$M:=J(F)_\Tor \otimes \Z[1/a]$ and $\tM:=\tJ(F)_\Tor \otimes \Z[1/a]$
according to characters of $\W$.
If we write $\tM^e$ and $M^e$ for the $e$-part
of $\tM$ and $M$ for $e \in \E_0$,
then they say
\begin{itemize}
\item 
$\tM^e=M^e=0$
if $e=1_{\E_0}$;
\item 
$\tM^e=0$
and 
$M^e
\cong (\Z/d(e)\Z) \otimes \Z[1/a]$
if $|e^{-1}(-1)|=1$;
\item 
$\tM^e
\cong M^e
\cong (\Z/d(e)\Z) \otimes \Z[1/a]$
%\cong \Z[1/a]/d(e)\Z[1/a]$
if $|e^{-1}(-1)|\ge 2$.
\end{itemize}
The referee pointed out that,
in view of \cite[Remark 4.4]{PW},
$\tM^e$ admits an interpretation as 
the kernel of the specialization map
$M^e \to \prod_{i=1}^s \Phi_{p_i}$,
where $\Phi_{p_i}$ is the group of components
of the reduction of $J_0(N)$ at $p_i$.
A direct proof of this statement 
would possibly lead to a more geometric proof
of Theorem \ref{thm:main2}.
\end{remark}

\begin{remark}
In Theorems \ref{thm:ohta-pw} and \ref{thm:main2},
it is sometimes possible to describe
the $3$-part (resp. $(q+1)$-part) 
in the case (NF) (resp.\ (FF)),
see Theorem \ref{thm:ohta-pw2} and Corollary \ref{cor:main3}.
\end{remark}

\subsection{Eisenstein property}\label{sec:intro-eis}
Let $\T$ be the $\Z$-algebra generated by the Hecke correspondences 
$\tau_p$ on $X$ for all $p \nmid N$. 
Then $J$ and $\widetilde{J}$ admit $\T$-module structures, 
and the natural homomorphism $\widetilde{J} \rightarrow J$ 
is actually $\T$-equivariant. 
Let $\sE$ be the ideal of $\T$
generated by $\tau_p -|p|-1$ for all $p \nmid N$.
In the case of (NF),
we also use a little smaller ideal $\sE'$ 
generated by $\tau_p -|p|-1$ for all $p \nmid 2N$.
We have the following  results in the literatures
(cf.\ Lemma \ref{lem:eis-j}).
\begin{itemize}
\item If $s=1$, then $J(F)_\Tor$ is annihilated by $\sE$.
\item In (NF), $J(F)_\Tor$ is annihilated by $\sE'$.
\item In (FF), $J(F)_\Tor \otimes \Z[1/q]$ is annihilated by $\sE$.
\end{itemize}
They played a fundamental role 
in the works of Mazur, P\'al, Ohta and Papikian-Wei
(cf.\ \cite[Proposition 11.1]{Mazur},
\cite[Lemma 7.16]{Pal},
\cite[p.\ 316]{Ohta},
\cite[Lemma 7.1]{PW2}).
We now ask an analogous question for $\tJ(F)_\Tor$.
In \S \ref{sect:eis},
we will study this problem and prove the following results
in Proposition \ref{P-Eisenstein}
(see also Corollary \ref{cor:prime-eis}).

\begin{prop}
${}$
\begin{enumerate}
\item
If $s=1$, then $\tJ(F)_\Tor$ is annihilated by $\sE$.
\item
In {\rm (NF)},
$\tJ(F)_\Tor$ is annihilated by ${\sE'}^2$.
\item
In {\rm (FF)},
$\tJ(F)_\Tor \otimes \Z[1/q]$ is annihilated by $\sE^2$.
\end{enumerate}
\end{prop}

\subsection{Organization of the paper}
%This article is organized as follows.
We recall known facts and prove easy lemmas on 
Jacobian varieties (resp.\ generalized Jacobians)
of (Drinfeld) modular curves
in \S \ref{sec:jac} (resp.\ \S \ref{sec:genjac}).
The proofs of
Theorems \ref{thm:main1} and \ref{thm:main2}
are reduced to a key technical result Theorem \ref{thm:main3},
whose proof occupies \S \ref{sec:pf}.
We discuss the Hecke action and the Eisenstein property of $\tJ(F)_\Tor$
in \S \ref{sect:eis}.

\section{Jacobian varieties}\label{sec:jac}
\subsection{Structure of cusps}
We continue to use the notation introduced in \S \ref{sec:intro-sqfree}.
Put $A_+=\Z_{>0}$ in the case (NF),
and let $A_+ \subset A$
be the set of all monic polynomials in the case (FF).
Recall that $C \subset X=X_0(N)$ denotes
the set of all cusps,
which admits a standard description
\begin{equation}
\label{eq:cusp1}
C \cong \Gamma_0(N) \backslash \P^1(F).
\end{equation}
For $x \in \P^1(F)$ we denote by $[x] \in C$
the point corresponding to the $\Gamma_0(N)$-orbit of $x$.
We shall constantly use two bijections
\begin{equation}\label{eq:w-a-c}
\W := (\Z/2\Z)^s
\overset{m}{\longrightarrow} 
\{ m \in A_+ \mid m|N \}
\overset{[(-)^{-1}]}{\longrightarrow} 
C .
\end{equation}
Here 
the first map is given by
$m(w)=\prod_{i=1}^s p_i^{\tilde{w}_i}$
for $w=(w_i)_{i=1}^s \in \W$,
where for $x \in \Z/2\Z$ we write
$\tilde{x}=0 \in \Z$ if $x=0$
and $\tilde{x}=1 \in \Z$ if $x=1$.
Note that $m$ becomes an isomorphism of groups
if we equip a group structure on 
$\{ m \in A_+ \mid m|N \}$
by $m * m' = mm'/(m, m')^2$.
The second map is given by 
sending $m$ to $[1/m]$.
%defining $[m] \in C$ to be the cusp corresponding to
%the $\Gamma_0(N)$-orbit of $1/m$ in $\P^1(F)$
%under the canonical bijection \eqref{eq:cusp1}.
%(Here in the case (FF)
%we denote by $\C$ the completion of an algebraic closure
%of the completion of $k$ at $t=\infty$.)
We shall abbreviate $[w]:=[1/m(w)]$ for $w \in \W$. 

We identify $\E:=\{ \pm 1 \}^s$  %$\E:=\Hom(\W, \G_m)=\{ \pm 1 \}^s$ 
with $\E_0$ from \eqref{eq:def-Ej}.
Let
\[ \langle , \rangle : \E \times \W \to \{ \pm 1 \} \]
be the following canonical biadditive pairing
$$\langle e, w \rangle = \prod_{i=1}^s e_i^{w_i} \in \{ \pm 1 \} \quad 
\text{for } e=(e_i)_{i=1}^s \in \E, ~w=(w_i)_{i=1}^s \in \W.
%\forall e=(e_i)_{i=1}^s \in \E \text{ and }w=(w_i)_{i=1}^s \in \W.
$$
Given $e=(e_i)_{i=1}^s \in \E$, we define
\begin{equation}\label{eq:def-De}
D^e := \sum_{w \in \W} \langle e, w \rangle[w] \in \Div(X).
\end{equation}
The degree of $D^e$ is zero if $e \not=1_\E$,
and is $2^s$ if $e=1_\E$.

\subsection{Atkin-Lehner involution}
For each $m \in A_+$ with $m|N$,
we let $W_m : X \to X$
be the Atkin-Lehner involution associated to $m$
(cf.\ \cite[(1.1.5)]{Ohta} in the case (NF), 
\cite[Definition 2.11]{PW} in the case (FF)).
To ease the notation, we write $W_w =W_{m(w)}$ for $w \in \W$.
Recall that $W_m$ restricts to 
an $F$-automorphism of $C$. % and of $Y=X \setminus C$.
\begin{lemma}\label{lem:d12}
\begin{enumerate}
\item
For $w, w' \in \W$ and $e \in \E \setminus \{ 1_\E \}$,
we have
\begin{align*}
W_w([w'])=[w+w'],
\quad
W_w(D^e)=\langle e, w \rangle D^e.
\end{align*}
\item 
Let us define subgroups of $\Div(X)$ by
\[ \sD_1:=\bigoplus_{e \in \E \setminus \{ 1_\E \}} \Z D^e
  \subset
   \sD_2:= \sD \cap \Div^0(X)
  \subset
   \sD:= \bigoplus_{w \in \W} \Z [w],
\]
%\[ \sD_1:=\bigoplus_{e \in \E \setminus \{ 1_\E \}} \Z D^e
%  \subset
%   \sD_2:=(\bigoplus_{w \in \W} \Z [w]) \cap \Div^0(X),
%\]
which are all stable under the action of $\W$.
Then the index $[\sD_2 : \sD_1]$ is given by $2^{(2^{s-1}-1)s}$.
\end{enumerate}
\end{lemma}
\begin{proof}
The first statement of (1) 
follows from the definition (cf.\ loc.\ cit.),
and the second follows from \cite[Proposition 4.2]{PW}.
To show (2), we consider a commutative diagram
with exact rows
\[
\xymatrix{
0 \ar[r] &
\sD_1 \ar[r] \ar@{^{(}-{>}}[d]_\alpha &
\displaystyle{\bigoplus_{e \in \E} \Z D^e}  \ar[r] \ar@{^{(}-{>}}[d]_\beta &
\Z D^{1_\E}  \ar[r] \ar@{^{(}-{>}}[d]_{\gamma} &
0
\\
0 \ar[r] &
\sD_2 \ar[r] &
\sD \ar[r]_-{\aug} &
%\displaystyle{\bigoplus_{w \in \W}} \Z w  \ar[r]_-{\aug} &
\Z \ar[r]  &
0,
}
\]
where $\alpha$ and $\beta$ are the inclusions,
$\aug$ is the augmentation,
and $\gamma$ is induced by $\aug \circ \beta$.
Since $\gamma$ is injective 
and its cokernel is cyclic of order $2^s$,
we are reduced to showing $|\Coker(\beta)|=2^{2^{s-1}s}$.
Define a matrix $A_s \in M_{2^s}(\Z)$ by
$A_s:=(\langle e, w \rangle)_{e \in \E, w \in \W}$
so that we have $|\Coker(\beta)|=|\det A_s|$.
Then we have
$A_s =
\begin{pmatrix} A_{s-1} & A_{s-1} \\ A_{s-1} & -A_{s-1} \end{pmatrix}$
for any $s \geq 1$.
Now the claim follows by induction.
\end{proof}

\begin{remark}\label{rem:d123}
%\begin{enumerate}
%\item 
%The projection map defines an isomorphism
%\[
%\sD_2 \cong
%\sD_2':= \bigoplus_{w \in \W \setminus \{ [\infty] \}} \Z[w],
%\]
%where $[\infty]=[N] \in \W$ is the cusp over $\infty$.
%\item 
The composition of canonical homomorphisms
\begin{equation}\label{eq:D3}
\sD_2 \hookrightarrow \Z[w] \twoheadrightarrow 
\sD_3:=\sD/\langle \sum_{w \in \W} [w] \rangle_\Z
%\sD_3:=(\bigoplus_{w \in \W} \Z [w])/\langle \sum_{w \in \W} [w] \rangle_\Z
\end{equation}
is injective and its cokernel is of order $2^s$.
Hence we have
\[ \sD_1 \otimes \Z\left[\frac{1}{2}\right] 
= \sD_2 \otimes \Z\left[\frac{1}{2}\right]
= \sD_3 \otimes \Z\left[\frac{1}{2}\right].
\]
For a $\Z[\W]$-module $M$ and $e \in \E$,
we write 
\begin{equation}\label{eq:e-part}
M^e = \{ x \in M \mid w x = \langle e, w \rangle x
~\text{for all}~  w \in \W \}.
\end{equation}
Then for any $\Z[1/2]$-module $B$ and $i=1, 2, 3$ we have 
\[ \sD_i \otimes B \cong 
\bigoplus_{e \in \E \setminus \{ 1 \}} (\sD_i \otimes B)^e
\]
and moreover
\begin{equation}\label{eq:D123}
(\sD_1 \otimes B)^e=(\sD_2 \otimes B)^e=(\sD_3 \otimes B)^e
=D^e \otimes B.
\end{equation}
%\end{enumerate}
\end{remark}

\subsection{Jacobian variety}\label{sect:jac-var}
Let $J$ be the Jacobian variety of $X$.
We define $\sC$ to be the cuspidal divisor subgroup, i.e.\
\begin{equation}\label{eq:def-sC0}
\sC := \IM(\sD_2 \to J(F)) \subset J(F)
\end{equation}
%to be the image of the canonical map $\sD_2 \to J(F)$,
where $\sD_2$ is from Lemma \ref{lem:d12}.
It is known that $\sC$ is contained in $J(F)_\Tor$
by \cite{Manin}, \cite{Gekeler}.
Since $J(F)_{\Tor}$ is finite,
we can decompose 
\begin{equation}\label{eq:def-sC}
J(F)_{\Tor} = \bigoplus_{\ell} J(F)\{\ell\},
\quad
\sC = \bigoplus_{\ell} \sC\{\ell\},
\end{equation}
where $\ell$ runs through all primes
and $\{ \ell \}$ denotes the $\ell$-primary torsion part.
If $\ell \not= 2$
we may further decompose 
(see \eqref{eq:e-part})
\[
J(F)\{\ell\} \cong \bigoplus_{e \in \E} J(F)\{\ell\}^e,
\quad
\sC\{\ell\} \cong \bigoplus_{e \in \E} \sC\{\ell\}^e.
\]

Suppose either that 
we are in the case (NF) and $(3,N)=1$,
or that we are in the case (FF).
We define $e_H \in \E$ by
\begin{align}\label{eq:def-eh}
e_H:=
\begin{cases}
\left((\frac{p_1}{3}), \dots, (\frac{p_s}{3})\right)
& \text{in (NF),~ $(3, N)=1$},
\\
\left((-1)^{\deg p_1}, \dots, (-1)^{\deg p_s}\right)
& \text{in (FF)}.
\end{cases}
\end{align}
Note that for $e \in \E$ we have
$(3, d(e))=3$ if and only if $e \not= e_H$
in the case (NF)
(see \eqref{eq:def-de)} for the definition of $d(e)$).
In the case (FF),
we have $(q+1, d(e))=q+1$ for any $e \not= e_H$,
and $(q+1, d(e_H))$ is a power of two
\cite[Remark 3.6]{PW}.
We also put
\begin{equation}
\label{eq:def-k}
k := 
\begin{cases}
12, 
\\
q^2-1, 
\end{cases}
\quad
b := 
\begin{cases}
3, & \text{ in (NF),}\\
q+1, & \text{ in (FF).}
\end{cases}
\end{equation}
In the case (FF), we write $|a|=q^{\deg a}$ for $a \in A \setminus \{ 0 \}$.
%Recall that we have defined $d(e) \in \Z_{>0}$ for $e \in \E$
%in \eqref{eq:def-de)}.

\begin{theorem}\label{thm:ohta-pw2}
${}$
\begin{enumerate}
\item (Mazur \cite[Theorem 1]{Mazur}; P\'al \cite[Theorems 1.2, 1.4]{Pal})
Suppose $s=1$. % (that is, $N=p_1 \in A$ is a prime element). 
Then we have $\sC=J(F)_\Tor$ and it is a cyclic group of order
$(|N|-1)/(k, |N|-1)$.
\item (Ohta, \cite[Theorem (3.6.2)]{Ohta}; Papikian-Wei, \cite[Theorem 4.3]{PW})
Let $\ell$ be an odd prime.
In {\rm (NF)} and $\ell=3$, we further assume $(N, 3)=1$.
If we are in {\rm (FF)}, assume $(\ell, q(q-1))=1$.
Then, for any $e \in \E$,
we have $\sC\{\ell\}^e=J(F)\{\ell\}^e$ and it is a cyclic group 
whose order is the $\ell$-part of 
\[
\begin{cases}
1 & \text{if}~ e=1_\E,
\\
d(e)
& \text{if}~ e = e_H \not= 1_\E,
\\
d(e)/b & \text{if}~ e \not=1_\E, e_H.
\end{cases}
\]
It is generated by the class of $D^e$
unless $e=1_\E$.
\end{enumerate}
\end{theorem} 

Note that we do not need $e_H^\pm$ appearing in \cite{Ohta},
because we assume $\ell \not= 2$ and $(3, N)=1$ if $\ell=3$

\section{Generalized Jacobian}\label{sec:genjac}

\subsection{An exact sequence}
Let $\tJ$ be the generalized Jacobian variety of $X$
with modulus $C$ (see \eqref{eq:cusp1} for $C$).
Here we quickly recall some basic results
from \cite{Serre} and \cite[\S 2]{YY},
with which the readers may consult for details.
We have (almost by definition)
\begin{equation}\label{eq:def-tj}
\tJ(F) = \Div^0(X \setminus C)/\{ \div(f) \mid f \in F(X)^\times,
~ f \equiv 1 \bmod C \}.
\end{equation}
Here by $f \equiv 1 \bmod C$ we mean 
$\ord_x(f-1)>0$ for all $x \in C$.
It follows that there is an exact sequence
\begin{equation}\label{eq:delta}
0 \to \sD_3 \otimes (F^\times)_\Tor
\to \tJ(F)_\Tor \to J(F)_\Tor
\overset{\delta}{\to} \sD_3 \otimes F^\times \otimes \Q/\Z,
\end{equation}
where $\sD_3$ is from \eqref{eq:D3}.
%(Note that \eqref{eq:exact1} and \eqref{eq:exact2}
%are parts of the special cases of this sequence.)
All maps in this sequence are compatible with
the action of Atkin-Lehner involutions.
In particular, 
we can decompose the map $\delta$,
up to $2$-primary torsion,
into the direct sum of 
\begin{equation}\label{eq:delta-dec}
\delta_\ell^e :
J(F)\{\ell\}^e \to D^e \otimes F^\times \otimes \Q_\ell/\Z_\ell,
\end{equation}
where $\ell$ and $e$ ranges over all odd primes
and elements of $\E$ respectively,
see \eqref{eq:D123}.

\subsection{Connecting map}
We recall a description of the map $\delta$.
Let $\sD_2$ and $\sD_3$ 
be groups defined in Lemma \ref{lem:d12} and Remark \ref{rem:d123},
respectively.
We shall identify
$\sD_3 \otimes F^\times \otimes \Q/\Z$
with the quotient of
\[
\bigoplus_{w \in \W} \Z[w] \otimes F^\times \otimes \Q/\Z
%= \bigoplus_{w \in \W} F^\times \otimes \Q/\Z
\]
by %the diagonal image of 
$\langle \sum_{w \in \W} [w] \rangle_\Z \otimes F^\times \otimes \Q/\Z$.
\begin{lemma}[{\cite[Lemma 2.3.1]{YY}}]\label{lem:delta}
Suppose that the class $[D] \in J(F)$
of $D=\sum_{w \in \W} a_w [w] \in \sD_2$
is killed by $m \in \Z_{>0}$
so that there is an $f \in F(X)^\times$
such that $\div(f)=mD$.
Then $\delta([D])$ is given by the image
in $\sD_3 \otimes F^\times \otimes \Q/\Z$ of
\begin{equation}\label{eq:lem-delta}
\sum_{w \in \W} [w] \otimes \frac{f}{t_{[w]}^{ma_w}}([w]) 
\otimes \frac{1}{m} \quad 
\in 
\bigoplus_{w \in \W} \Z[w] \otimes F^\times \otimes \Q/\Z,
\end{equation}
where $t_{[w]}$ is a uniformizer at $[w] \in C$.
(Note that the image of 
\eqref{eq:lem-delta} 
in $\sD_3 \otimes F^\times \otimes \Q/\Z$ 
depends only on $[D] \in J(F)$
and is independent of the choices
of $m$, $f$ and $t_{[w]}$.)
\end{lemma}

\subsection{Main result}
For $i = 1,...,s$, 
define $e^{(i)}=(e^{(i)}_j)_j \in \E$ by 
\begin{equation}\label{eq:def-ei}
e^{(i)}_j = 
\begin{cases}
-1 & \text{ if } i=j,\\
1 & \text{ otherwise.}
\end{cases}
\end{equation}
We now arrive at our main result
(see \eqref{eq:def-k} and \eqref{eq:def-de)}
for the definitions of $k$ and $d(e)$).

\begin{theorem}\label{thm:main3}
Let $e \in \E \setminus \{ 1_\E \}$.
The order of $\delta([D^e])$ is given by
\[
\begin{cases}
d(e^{(i)})/(d(e^{(i)}),~ 2^{s-1}k)
& \text{ if $e=e^{(i)}$ for some $i$},
\\
1 
& \text{ otherwise.}
\end{cases}
\]
\end{theorem} 

The proof of this theorem is given in the next section.
Combined with Theorem \ref{thm:ohta-pw2},
it implies Theorems \ref{thm:main1} and \ref{thm:main2}.
For the latter, we can also deduce 
the following supplementary results.
We define
\[
\sM_2':= \bigoplus_e
\left( \Z/(d(e)/b) \Z \right).
\]
where $e$ ranges over 
$\E \setminus \{ 1_\E, e_H, e^{(1)}, \dots, e^{(s)} \}$
(see \eqref{eq:def-k} and \eqref{eq:def-eh} for $b$ and $e_H$).

\begin{corollary}\label{cor:main3}
If we are in the case {\rm (NF)} 
we assume $(3, N)=1$ and put $\ell=3$.
If we are in the case {\rm (FF)} 
we assume  $\ell$ is an odd prime divisor of $q+1$.
Then there is an isomorphism
\[ \tJ(F)\{\ell\}
\cong \sM_2' \otimes \Z_{(\ell)}.
\]
\end{corollary}
\begin{proof}
Using remarks after \eqref{eq:def-eh} and
$\ord_\ell(k)=\ord_\ell(b)$,
this follows by comparing 
Theorem \ref{thm:ohta-pw2} with Theorem \ref{thm:main3}.
\end{proof}

\section{Proof of Theorem \ref{thm:main3}}\label{sec:pf}

\subsection{Discriminant functions}\label{sec:disc}
Let $F_\infty$ be the completion of $F$ with respect to the absolute value $|\cdot|$ on $F$
(which means $\Q_\infty = \R$ and $\F_q(t)_\infty = \F_q(\!(t^{-1})\!)$).
Let $\C_\infty$ be the completion of a chosen algebraic closure of $F_\infty$.
Let 
%$\mathfrak{H}$ be the (resp.\ Drinfeld) upper half plane, i.e.\
%the (resp.\ Drinfeld) upper half plane (resp.\ in the case (FF)), i.e.\
$$\mathfrak{H} := \begin{cases} \{z \in \C_\infty \mid \text{Im}(z)>0\}, & \text{ in the case (NF),} \\
\C_\infty - F_\infty, & \text{ in the case (FF).}
\end{cases}$$
Let $\Delta(z)$ be the 
\textit{modular discriminant function} \cite[Example 3.4.3]{Silverman}
(resp.
\textit{Drinfeld discriminant function} \cite[(1.2)]{Gekeler})  on $\mathfrak{H}$,
which satisfies
$$\Delta\left(\frac{az+b}{cz+d}\right) = (cz+d)^k \cdot \Delta(z)
\quad \text{ for all } 
\begin{pmatrix}a&b\\ c&d\end{pmatrix} \in G,
%\quad \forall \begin{pmatrix}a&b\\ c&d\end{pmatrix} \in \SL_2(A) \text{ (resp.\ $\GL_2(A)$)},
$$
where $k$ is from \eqref{eq:def-k} and
$G=\SL_2(A)$ (resp.\ $G=\GL_2(A)$) 
in the case (NF) (resp.\ (FF)).
%
%\begin{equation}
%\label{eq:def-k}
%k := 
%\begin{cases}
%12, & \text{ in the case (NF),}\\
%q^2-1, & \text{ in the case (FF).}
%\end{cases}
%\quad
%G := 
%\begin{cases}
%\SL_2(A), \\
%\GL_2(A);
%\end{cases}
%\end{equation}
Given $e=(e_i)_{i=1}^s \in \E$, we define
\begin{align}
\label{eq:def-delta-e}
\Delta^e(z) := \prod_{w \in \W} \Delta(m(w)z)^{\langle e, w \rangle}.
%\\ &d(e):=\prod_{i=1}^s(|p_i|+e_i) \in \Z_{>0}.
\end{align}

For $w \in \W$ and $e \in \E \ \backslash \{1_\E\}$, 
we define 
\begin{equation}\label{eq:def-c}
c(e, w)=
\begin{cases} p_i^{-2^{s-1}  k} & \text{ if $e = e^{(i)}$ and $w_i =1$,} \\
1 & \text{ otherwise},
\end{cases}
\end{equation}
where $k$ is from \eqref{eq:def-k}
and $e^{(i)}$ is from \eqref{eq:def-ei}.

\begin{lemma}\label{lem 4.1.1}
Given $w \in \W$ and $e \in \E \ \backslash \{1_\E\}$, 
we have
\begin{eqnarray}\label{eqn 2.2}
\Delta^e( W_w z) &=&c(e, w) \Delta^e(z)^{\langle e,w\rangle}.
\end{eqnarray}
\end{lemma}
\begin{proof}
From the transformatin law of $\Delta$, it is straightforward that
$$
\Delta^e (W_w z) = \Delta^e(z)^{\langle e, w\rangle} \cdot \prod_{w' \in \W} \big(m(w'),m(w)\big)^{-\langle e, w' \rangle \cdot k}.$$
For $i = 1,...,s$, one has
(see a sentence after \eqref{eq:w-a-c} for the notation $\tilde{} ~$)
\begin{align*}
\ord_{p_i}&\left( \prod_{w' \in \W} \big(m(w'),m(w)\big)^{\langle e, w' \rangle}\right)
= \tilde{w}_i \cdot \sum_{w' \in \W \atop w'_i = 1} \langle e, w'\rangle \nonumber %=c(e, w),
\\
&=
\begin{cases} 2^{s-1}, & \text{ if $e = e^{(i)}$ and $w_i = 1$,} \\
0, & \text{ otherwise},
\end{cases}
\end{align*}
%
%
%\begin{eqnarray}
%\ord_{p_i}\left( \prod_{w' \in \W} \big(m(w'),m(w)\big)^{\langle e, w' \rangle}\right)
%&=& \tilde{w}_i \cdot \sum_{w' \in \W \atop w'_i = 1} \langle e, w'\rangle \nonumber %=c(e, w),
%\\
%&=&
%\begin{cases} 2^{s-1}, & \text{ if $e = e^{(i)}$ and $w_i = 1$,} \\
%0, & \text{ otherwise.}
%\end{cases} \nonumber
%\end{eqnarray}
proving the lemma.
\end{proof}

\begin{lemma}
Let $e \in \E \setminus \{ 1_\E \}$.
\begin{enumerate}
\item 
The function $\Delta^e(z)$ from \eqref{eq:def-delta-e}
is a rational function on $X$ defined over $F$.
\item 
We have
$\div(\Delta^e)=d(e) D^e$
(see \eqref{eq:def-de)} and \eqref{eq:def-De}).
\end{enumerate}
\end{lemma}

\begin{proof}
(1) The transformation law of $\Delta$ implies that $\Delta^e$ lies in $\C_\infty(X)$.
Note that the \lq\lq $q$-expansion of $\Delta$ at $\infty$\rq\rq\ has $F$-rational coefficients (cf.\ \cite{Gekeler2} in the case (FF)). Regarding $\C_\infty(X)$ as a subfield of $\C_\infty(\!(q)\!)$, the result holds from the fact that
$F(X) = \C_\infty(X) \cap F(\!(q)\!)$.\\
(2) Let $w_\infty:= (-1,...,-1) \in \W$. Then $[w_\infty] = [1/N] \in C$, corresponding to the \lq\lq cusp at $\infty$\rq\rq. Considering the $q$-expansion of $\Delta$, it is observed that
$$\ord_{[w_\infty]}(\Delta^e) = d(e) \cdot \langle e, w_\infty\rangle  = d(e) \cdot \ord_{[w_\infty]} D^e.$$
Thus the result follows from Lemma~\ref{lem:d12} (1) and \ref{lem 4.1.1}.
\end{proof}

%{\bf 
%\eqref{eq:def-delta-e}
%depends on the analytic description
%$G \backslash(\C_\infty \setminus F_\infty) \cong Y(\C_\infty)$.
%We need to justify them.
%To show it is defined over $F$,
%it is probably sufficient to show its
%$q$-expansion has $F$-coefficients.
%Indeed, 
%if we regard $\C_\infty(X)$ as a subfield of $\C_\infty((q))$ by $q$-expansion
%(and $F((q)) \subset \C_\infty((q))$ by extending $F \subset \C_\infty$),
%we have $F(X) = \C_\infty(X) \cap F((q))$
%because $X$ is absolutely integral over $F$.
% }

\subsection{Evaluation of the connecting map}
%In view of Theorem \ref{thm:ohta-pw2},
Theorem \ref{thm:main3} immediately follows from the following result:

%We consider the map (see \eqref{eq:D3} Lemma \ref{lem:delta})
%\[ \delta : J(F)_{\Tor} \to \sD_3 \otimes_\Z F^\times \otimes_\Z \Q/\Z .
%\]
%({\bf Need to write a connection with \eqref{eq:exact1}, \eqref{eq:exact2}})\\
%
%For $i=1, \dots, s$,
%we define $e^{(i)}=(e_j^{(i)}) \in \E$ by
%$e_j^{(i)}=1$ for $j \not= i$ and $e_i^{(i)}=-1$.
%
\begin{proposition}\label{prop:order}
%Recall from \eqref{eq:def-k}
%that we put $k=12$ in the case (NF)
%and $k=q^2-1$ in the case (FF).
Let $e \in \E \setminus \{ 1_\E \}$.
If $e=e^{(i)}$ for some $i=1, \dots, s$
(see \eqref{eq:def-ei}), then
we have
%\begin{eqnarray}
%\delta(D^{e^{(i)}})
%= D^{e^{(i)}} \otimes p_i \otimes
%\frac{-2^{s-1} k}{d(e^{(i)})} \quad \text{ in }  \sD_3 \otimes_\Z F^\times \otimes_\Z \Q/\Z. \nonumber
%\end{eqnarray}
\begin{eqnarray}
\delta(D^{e^{(i)}})&=&
\sum_{w \in \W \atop w_i = 0} [w] \otimes p_i \otimes \frac{-2^{s-1} k}{d(e^{(i)})} \nonumber \\
&=& D^{e^{(i)}} \otimes p_i \otimes
\frac{-2^{s-2} k}{d(e^{(i)})} \quad \text{ in }  \sD_3 \otimes F^\times \otimes \Q/\Z. \nonumber
\end{eqnarray}
Otherwise, we have $\delta(D^e)=0$.
\end{proposition}

\begin{proof}
The second equality follows from
$$D^{e^{(i)}} = \sum_{w \in \W \atop w_i = 0}[w] - \sum_{w \in \W \atop w_i = 1}[w]
= 2 \sum_{w \in \W \atop w_i = 0}[w] - \sum_{w \in \W}[w] 
\quad \text{ in } 
\bigoplus_{w \in \W} \Z[w].$$
To show the first equality, 
by Lemma \ref{lem:delta} it suffices to take suitable uniformizers $t_{[w]} \in F(X)$ for $w \in \W$ so that
\begin{equation}\label{eq:final}
\frac{\Delta^e}{t_{[w]}^{\langle e,w \rangle \cdot d(e)}}([w])=c(e, \bar{w})
\quad \text{for all } e \in \E\ \backslash\ \{1_\E\}, w \in \W,
\end{equation}
where $c(e, \bar{w})$ is from \eqref{eq:def-c}
and 
we put $\bar{w} := w_\infty - w (=w_\infty + w) \in \W$
with $w_\infty := (1,...,1) \in \W$.
We first take $t_{[w_\infty]} \in F(X)^\times$ 
to be any uniformizer at $w_\infty$
which has $1$ as the leading term of its $q$-expansion.
(For instance, we may take 
the pull-back of the reciprocal of the $j$-function
$X_0(1) \overset{\cong}{\longrightarrow} \P^1$
along the morphism $\pi : X=X_0(N) \to X_0(1)$ 
given by forgetting the level structure.)
Then we have 
%(e.g.\ take $t_{[w_\infty]} = j^{-1}$, the inverse of the $j$-function) 
$$\left(\frac{\Delta^e}{t_{[w_\infty]}^{\langle e,w_\infty\rangle \cdot d(e)}}\right) ([w_\infty]) = 1
$$
for any $e \in \E \setminus \{ 1_\E \}$.
For general $w \in \W$, let
$$t_{[w]} := t_{[w_\infty]} \circ W_{\bar{w}} \in F(X).$$
Then we have
\[
\left(\frac{(\Delta^e)^{\langle e,\bar{w}\rangle}}{t_{[w_\infty]}^{\langle e,w\rangle \cdot d(e)}}\right) ([w_\infty])
=
%\left(\frac{(\Delta^e)^{\langle e,\bar{w}\rangle}}{t_{[w_\infty]}^{\langle e,w_\infty\rangle \cdot d(e) \cdot \langle e,\bar{w}\rangle}}\right) 
%\left(\frac{t_{[w_\infty]}^{\langle e,w_\infty\rangle \cdot d(e) \cdot \langle e,\bar{w}\rangle}}{t_{[w_\infty]}^{\langle e,w\rangle \cdot d(e)}}\right) 
%([w_\infty])
\left(\frac{\Delta^e}{t_{[w_\infty]}^{\langle e,w_\infty\rangle \cdot d(e)}}\right)^{\!\!\! \langle e, \bar{w}\rangle} \!\!\! ([w_\infty])
=1.
\]
%Choose a uniformizer $t_{[w_\infty]} \in F(X)$ at the infinite cusp $[w_\infty]$, 
%and for every $w \in \W$ we let
%$$t_{[w]} := t_{[w_\infty]} \circ W_{\bar{w}} \in F(X).$$
Therefore for any $e \in \E\ \backslash\ \{1_\E\}$ and $w \in \W$,
applying \eqref{eqn 2.2}, we obtain
\begin{eqnarray}
\left(\frac{\Delta^e}{t_{[w]}^{\langle e,w\rangle \cdot d(e)}}\right) ([w])
&=&\left(\frac{\Delta^e \circ W_{\bar{w}}}{t_{[w_\infty]}^{\langle e,w\rangle \cdot d(e)}}\right) ([w_\infty]) \nonumber \\
&=&c(e, \bar{w}) \cdot
\left(\frac{(\Delta^e)^{\langle e,\bar{w}\rangle}}{t_{[w_\infty]}^{\langle e,w\rangle \cdot d(e)}}\right) ([w_\infty]) \nonumber \\
&=& c(e, \bar{w}),  \nonumber %\quad \text{for all } w \in \W, \nonumber
\end{eqnarray}
and \eqref{eq:final} holds.
\end{proof}

%We may decompose $\delta$ into a direct sum of
%\[ \delta_\ell : J(F)\{ \ell \} \to \sD_3 \otimes_\Z F^\times \otimes_\Z \Q_\ell/\Z_\ell,
%\]
%where $\ell$ ranges over all primes.
%Since the map $\delta$ respects Atkin-Lehner involutions,
%if $\ell \not= 2$
%we may further decompose $\delta_\ell$ into
%a direct sum of (see \eqref{eq:D123})
%\[ \delta_\ell^e : J(F)\{ \ell \}^e \to D^e \otimes_\Z F^\times \otimes_\Z \Q_\ell/\Z_\ell
%\]
%where $e$ ranges over $\E \setminus \{ 1_\E \}$.
%The above proposition therefore determines precisely the image of $\delta_\ell^e$ for every prime $\ell$ satisfying the assumptions in Theorem \ref{thm:ohta-pw2}.

\subsection{Generators of $\sC \cap \ker \delta$}

For $i=1,...,s$, we have
$$2^{s-1}\cdot ([1]-[1/p_i]) = \sum_{e \in \E \atop e_i = -1} D^e.$$
Let $d(\E,i) := \prod_{e \in \E, e_i = -1} d(e)$. Then
$$d(\E,i) 2^{s-1} \cdot ([1]-[1/p_i]) =  \div\left(  \prod_{e \in \E \atop e_i = -1}(\Delta^e)^{\frac{d(\E,i)}{d(e)}}\right) \quad \text{ in } \Div^0(X).$$
Thus the equalities \eqref{eq:lem-delta} and \eqref{eq:final} imply
\begin{eqnarray}\label{eqn 4.5}
\delta([1]-[1/p_i])&=&
\sum_{w \in \W \atop w_i = 0} [w] \otimes p_i^{-2^{s-1}k \cdot \frac{d(\E,i)}{d(e^{(i)})}} \otimes \frac{1}{2^{s-1}d(\E,i)} \nonumber \\
&=& \sum_{w \in \W \atop w_i = 0} [w] \otimes p_i \otimes \frac{- k}{d(e^{(i)})}  \\
&=& D^{e^{(i)}} \otimes p_i \otimes
\frac{- k}{2d(e^{(i)})} \quad \text{ in }  \sD_3 \otimes F^\times \otimes \Q/\Z,  \nonumber
\end{eqnarray}
which is of order $d(e^{(i)})/(d(e^{(i)}), k)$.
In particular, for $m \mid (N/p_i)$ we get
\begin{eqnarray}\label{eqn 4.6}
\delta([1/m] - [1/(mp_i)]) &=& \delta\big(W_m([1]-[1/p_i])\big) \nonumber \\
&=& W_m(D^{e^{(i)}}) \otimes p_i \otimes
\frac{- k}{2d(e^{(i)})} \nonumber \\
&=& D^{e^{(i)}} \otimes p_i \otimes
\frac{- k}{2d(e^{(i)})} \ = \ \delta([1]-[1/p_i]). 
\end{eqnarray}
Observe that 
$$\Big\{[1/m]-[1/(mp_i)] \ \Big|\ i=1,...,s,\ \text{ and } m \mid (N/p_1\cdots p_i)\Big\}$$ 
is a $\Z$-basis of $\sD_2$.
Therefore equalities \eqref{eqn 4.5} and \eqref{eqn 4.6} lead us to:

\begin{corollary}\label{cor: ker}
For $i = 1,...,s$ and $m \mid (N/p_i)$, we have 
$$\delta\big([1]-[1/p_i] - [1/m] + [1/(mp_i)]\big) = 0.$$
Moreover, the intersection of $\sC \cap \ker \delta$ is generated by
$$\Big\{[1]-[1/p_i] - [1/m] + [1/(mp_i)] \ \Big|\ i = 1,...,s \text{ and } m \mid (N/p_1\cdots p_i)\Big\}$$
and
$$\Big\{\frac{d(e^{(i)})}{(d(e^{(i)}),k)} \cdot ([1]-[1/p_i]) \ \Big| \ i = 1,...,s\Big\}.$$
\end{corollary}

\section{Eisenstein property}\label{sect:eis}

\subsection{Jacobian variety}
Let $\T$ be the polynomial ring $\Z[\tau_p : p \nmid N]$ 
generated by variables $\tau_p$ for 
all prime elements $p \in A_+$ such that $p \nmid N$.
Let $\mathcal{E}$ 
be the ideal generated by $\tau_p-|p|-1$ for all $p \nmid N$,
which is called the {\it Eisenstein ideal}.
We say a $\T$-module is {\it Eisenstein}
if it is annihilated by $\mathcal{E}$.
In the case (NF),
we define $\mathcal{E}'$ 
to be the ideal generated by $\tau_p-p-1$ for all $p \nmid 2N$,
and we say a $\T$-module is
{\it Eisenstein away from $2$}
if it is annihilated by $\mathcal{E}'$.
When $N$ is even,
this is the same as saying Eisenstein.

Given $p \nmid N$, 
the Hecke correspondence on $X$ associated to $p$ 
induces an endomorphism $T_p$ of $J$.
%and we have $T_p T_{p'} =T_{p'} T_{p}$ for any $p, p' \nmid N$.
We obtain a $\Z$-algebra homomorphism
$\T \rightarrow \End(J)$ sending $\tau_p$ to $T_p$.
Then we may ask if $J(F)_\Tor$ is Eisenstein.
We collect known results in the literatures.

\begin{lemma}\label{lem:eis-j}
${}$
\begin{enumerate}
\item In {\rm (NF)},
$J(F)_\Tor$ is Eisenstein away from $2$, and
$J(F)_{\Tor} \otimes \Z[1/2]$ is Eisenstein.
\item In the case {\rm (FF)},
$J(F)_{\Tor} \otimes \Z[1/q]$ is Eisenstein.
\item 
In general,  $\sC$ is Eisenstein
(see \eqref{eq:def-sC0}). 
Consequently,
$J(F)_\Tor$ is Eisenstein
if $s=1$
by {\rm Theorem \ref{thm:ohta-pw2} (1)}.
\end{enumerate}
\end{lemma}

%\begin{remark}
%In Lemma \ref{lem:eis-j} (1), 
%if $N$ is even, then $J(F)_{\Tor}$ is Eisenstein
%because in this case being 
%Eisenstein away from $2$ is the same as being Eisenstein.
%\end{remark}

\begin{proof}
(1) is shown in the proof of \cite[Theorem 3.6.2 (p.\ 316)]{Ohta}.
For the completeness sake, we sketch its proof.
If $\ell$ is an odd prime (resp.\ $\ell=2$),
then the reduction map 
$J(F) \to J_{/\F_\ell}(\F_\ell)$
restricted to $J(F)_\Tor$
(resp.\ $J(F)_\Tor \otimes \Z[1/2]$)
is injective 
by \cite[Appendix]{Katz}
(resp.\ \cite[IV, Proposition 3.1 (b)]{Silverman0}),
where $J_{/\F_\ell}$ is the reduction of $J$ over $\F_\ell$.
On the other hand,
the Eichler-Shimura congruence relation
shows that
the Hecke correspondence $T_\ell$ acts on $J_{/\F_\ell}(\F_\ell)$
by the multiplication by $\ell+1$,
whence the statement.

(2) is shown in \cite[Lemma 7.1]{PW2}.
(3) is well-known,
and it also follows from 
Proposition \ref{prop:sc-eisen} (3) below.
\end{proof}

\subsection{Generalized Jacobian}\label{sec:Hecke-genjac}
We can play a similar game for $\tJ$.
Given $p \nmid N$, 
the Hecke correspondence on $X$ associated to $p$ 
again induces an endomorphism $\tT_p$ of $\tJ$
as follows.
Let $f, g : X_0(pN) \to X_0(N)=X$ be the maps
that send a pair $(E, Q)$ of an elliptic curve 
(or a Drinfeld module) $E$
and its $pN$-torsion subgroup (submodule) $Q$
to the pair $(E, Q[N])$ 
and $(E/Q[p], Q/Q[p])$, respectively.
(We write by
$Q[N]$ the $N$-torsion subgroup (submodule) of $Q$,
and similarly for $Q[p]$.)
Over $\C_\infty$,
they are induced by the identity map 
and the multiplication by $p$
on $\P^1(\C_\infty) \sqcup \P^1(F)$,
passing through the quotients by 
$\Gamma_0(N)$ and $\Gamma_0(pN)$.
Recall that $C$ is the set of all cusps on $X=X_0(N)$.
Let $C'$ be the set of all cusps on $X_0(pN)$
and $\tJ'$ the generalized Jacobian of $X_0(pN)$
with modulus $C'$.
Since we have
$f^{-1}(C)=g^{-1}(C)=C'$,
there are the pull-back map $f^* : \tJ \to \tJ'$ 
and the push-forward map $g_* : \tJ' \to \tJ$.
The action of the Hecke correspondence is given by
\[ \tT_{p}: =g_{*} \circ f^* : \tJ \to \tJ. \]
We omit the proof of the following lemma,
which is well-known
(and is seen by the same way as $J$)

\begin{lemma}
We have $\tT_{p_1} \tT_{p_2} =\tT_{p_2} \tT_{p_1}$ for any $p_1, p_2 \nmid N$.
\end{lemma}
%\begin{proof}
%We may assume $p_1 \not= p_2$.
%We consider a diagram
%\[
%\xymatrix{
%X_0(p_1N) \ar[dd]_{g_1}
%& 
%& X_0(p_1p_2N) \ar[dl]_{g_{12}} \ar[dr]^{f_{12}} 
%\ar[ll]_{g_{21}} \ar[rr]^{f_{21}}
%& 
%& 
%X_0(p_2N) \ar[dd]^{f_2}
%\\
%& X_0(p_2N) \ar[dl]_{g_2} \ar[dr]^{f_2}
%& (*)
%& X_0(p_1N) \ar[dl]_{g_1} \ar[dr]^{f_1}
%&
%\\
%X_0(N)
%&
%& X_0(N)
%&
%& X_0(N).
%}
%\]
%Here the maps are defined by the same way as $f, g$ as above.
%Then all squares are commutative,
%and moreover the square $(*)$ is Cartesian.
%Therefore we have
%\begin{align*}
%T_{p_2} T_{p_1} &= 
%g_{2*}f^{*}_2g_{1*}f^*_{1}
%\overset{(**)}{=} 
%g_{2*}g_{12*}f^*_{12}f^*_{1}
%= g_{1*}g_{21*}f^*_{21}f^*_{2}
%\\
%&\overset{(**)'}{=} 
%g_{1*}f^*_{1}g_{2*}f^{*}_2
%=T_{p_1} T_{p_2}.
%\end{align*}
%Here $(**)$ holds since $(*)$ is Cartesian,
%and $(**)'$ by a symmetric argument.
%We are done.
%\end{proof}

Hence we obtain a $\Z$-algebra homomorphism
$\T \rightarrow \End(\tJ)$ sending $\tau_p$ to $\tT_p$,
and the canonical map $\tJ \to J$ is $\T$-equivariant.
Then we may ask if $\tJ(F)_\Tor$ is Eisenstein.
We provide a partial answer to this problem.
The first result is the following.

\begin{lemma}\label{P-Eisenstein0}
${}$
\begin{enumerate}
\item
In {\rm (NF)},
$\tJ(F)_{\Tor} \otimes \Z[1/2]$ is Eisenstein.
\item
In {\rm (FF)},
$\tJ(F)_\Tor \otimes \Z[1/q(q-1)]$ is Eisenstein.
\end{enumerate}
\end{lemma}
\begin{proof}
The Hecke algebra $\T$ acts on $\Div(X)$ as algebraic correspondences,
which induces $\T$-module structures on
the subgroup $\sD_2 \subset \Div(X)$ from Lemma \ref{lem:d12} (2),
and, in turn, on the quotient $\sD_3$ of $\sD_2$ from \eqref{eq:D3}.
Under the identification $\mu_F^{\oplus(2^s-1)}=\sD_3 \otimes \mu_F$,
the maps in the exact sequence \eqref{eq:exact2} are $\T$-equivariant.
Now the lemma immediately follows from Lemma \ref{lem:eis-j}.
\end{proof}

%\begin{prop}\label{P-Eisenstein}
%\begin{enumerate}
%\item
%In the case (NF),
%$\tJ(F)_\Tor$ is double Eisenstein away from $2$
%and
%$\tJ(F)_{\Tor} \otimes \Z[1/2]$ is Eisenstein.
%\item
%In the case (FF),
%$\tJ(F)_\Tor \otimes \Z[1/q]$ is double Eisenstein,
%and
%$\tJ(F)_\Tor \otimes \Z[1/q(q-1)]$ is Eisenstein.
%\item If $\sC=J(F)_\Tor$,
%then $\tJ(F)_\Tor$ is Eisenstein.
%In particular, this  holds when $s=1$
%by Theorem \ref{thm:ohta-pw2} (1).
%\end{enumerate}
%\end{prop}

We say a $\T$-module has {\it Eisenstein exponent two}
(resp.\ {\it Eisenstein exponent two away from $2$} in the case (NF))
if it is annihilated by $\sE^2$ (resp.\ ${\sE'}^2$).
As before, two notions are identical when $N$ is even.

\begin{prop}\label{P-Eisenstein}
${}$
\begin{enumerate}
\item
If $s=1$, then $\tJ(F)_\Tor$ is Eisenstein.
\item
In {\rm (NF)},
$\tJ(F)_\Tor$ has Eisenstein exponent two away from $2$.
\item
In {\rm (FF)},
$\tJ(F)_\Tor \otimes \Z[1/q]$ has Eisenstein exponent two.
\end{enumerate}
\end{prop}

%By Theorem \ref{thm:main1} and Lemma \ref{lem:eis-j},
%this proposition is reduced to the following lemma 
%that shows
%$\mu_{F}^{\oplus (2^s-1)}=\sD_3 \otimes \mu_F$ is Eisenstein
%in \eqref{eq:exact2},
%where $\sD_3$ is from \eqref{eq:D3}.

\begin{proof}
In general, 
if $0 \to M' \to M \to M''$ is an exact sequence
of $\T$-modules
and if $M', M''$ are Eisenstein (away from $2$),
then $M$ has Eisenstein exponent two (away from $2$).
Hence by Lemma \ref{lem:eis-j},
the proposition is reduced 
to the following lemma 
that shows
$\mu_{F}^{\oplus (2^s-1)}=\sD_3 \otimes \mu_F$ is Eisenstein
in \eqref{eq:exact2},
where $\sD_3$ is from \eqref{eq:D3}.
\end{proof}

\begin{lemma}\label{lem:d3-eis}
The group $\sD_3$  is Eisenstein.
\end{lemma}

We will prove this lemma
as a part of Proposition \ref{prop:sc-eisen} (3) below
(although it can also be shown directly).
To state the result,
we need some preparation. We define
\begin{align*}
&L := \bigoplus_{w \in \W} F(X)_{[w]}^\times/U_{[w]}^{(1)},
\quad
U_{[w]}^{(1)}:=\{ f \in F(X)_{[w]}^\times \mid \ord_{[w]}(f-1)>0 \},
\\
&L^0 := \ker(d : L \to \Z),
\quad
d((f_w)_w)=\sum_w \ord_{[w]}(f_w),
\end{align*}
where $F(X)_{[w]}$ is the completion of $F(X)$ at $[w] \in X$
and $\ord_{[w]} : F(X)_{[w]}^\times \to \Z$ 
is the normalized discrete valuation.
By definition, we have an exact sequence
(see Lemma \ref{lem:d12} (2) for $\sD$ and $\sD_2$)
\begin{equation}\label{eq:L0}
0 \to \sD \otimes F^\times \to L^0 \to \sD_2 \to 0. \end{equation}
Using the approximation lemma,
\eqref{eq:def-tj} can be rewritten as
\begin{equation}
\tJ(F) = \frac{\ker[(\deg, d) : \Div(X \setminus C) \oplus L 
\overset{}{\longrightarrow} \Z]}
{\{(\div_{X \setminus C}(f), \Delta(f)) \mid f \in F(X)^\times \}},
\end{equation}
where $\Delta : F(X)^\times \to L$ is induced by the diagonal embedding
(see \cite[\S 2.2]{YY} for more details).
We define $\widetilde{\sC} \subset \tJ(F)$
to be the image of the composition
of the canonical maps $L^0 \hookrightarrow L \to \tJ(F)$.
This is related to the cuspidal divisor class $\sC$
from \eqref{eq:def-sC0} by an exact sequence
\begin{equation}\label{eq:ex-ct}
0 \to \sD_3 \otimes F^\times \to \widetilde{\sC} \to \sC \to 0.
\end{equation}

\begin{prop}\label{prop:sc-eisen}
${}$
\begin{enumerate}
\item 
In {\rm (NF)}, $\tC$ is Eisensetein away from $2$ and has Eisenstein exponent two.
It is Eisenstein if $s=1$ (or if $N$ is even).
\item 
In {\rm (FF)}, $\tC$ is Eisensetein.
\item
In any case, $\sD_2, \sD_3$ and $\sC$ are Eisentein.
\end{enumerate}
\end{prop}

We obtain the following
because
$\sC=J(F)_\Tor$ implies
$\tJ(F)_\Tor \subset \widetilde{\sC}$.

\begin{cor}\label{cor:prime-eis}
Suppose $\sC=J(F)_\Tor$.
\begin{enumerate}
\item 
In {\rm (NF)}, $\tJ(F)_\Tor$ is Eisensetein away from $2$ and 
has Eisenstein exponent two.
\item 
In {\rm (FF)}, $\tJ(F)_\Tor$ is Eisensetein.
\end{enumerate}
\end{cor}

%\begin{remark}
%In (NF), $\tC$ is not Eisenstein
%if $N$ is odd and $s>1$.
%This will be seen from the proof below.
%\end{remark}

It remains to prove Proposition \ref{prop:sc-eisen}.
Take $p \nmid N$. 
We use the notations introduced in 
the beginning of \S \ref{sec:Hecke-genjac}.
For any $w \in \W$,
we define two cusps $[w]', [w^*]' \in C'$ on $X'=X_0(pN)$ 
to be the $\Gamma_0(pN)$-orbits of 
$1/m(w)$ and $1/(pm(w)) \in \P^1(F)$, respectively
(see \eqref{eq:w-a-c} for $m(w)$).
We have $f^{-1}([w])=g^{-1}([w])=\{ [w]', [w^*]' \}$.
The action of $\tau_p$ on $L$ is induced by
the direct sum of
\begin{equation}
\phi_w : 
\frac{F(X)^\times_{[w]}}{U_{[w]}^{(1)}}
\overset{f^*}{\to}
\frac{F(X')_{[w]'}^\times}{U_{[w]'}^{(1)}} \oplus 
\frac{F(X')_{[w^*]'}^\times}{U_{[w^*]'}^{(1)}}
\overset{g_*}{\to}
\frac{F(X)^\times_{[w]}}{U_{[w]}^{(1)}}
\end{equation}
over all $w \in \W$.
We remark that 
both of $f$ and $g$ are of degree $|p|+1$
and their ramification indexes 
are given by 
$e_f([w]')=|p|, ~e_f([w^*]')=1$
and 
$e_g([w]')=1, ~e_g([w^*]')=|p|$.
The following is the key step in the proof of
Proposition \ref{prop:sc-eisen}.

\begin{lemma}\label{finallemma}
For any $\alpha \in F(X)^\times_{[w]}$,
we have 
\begin{equation}\label{eq:finallemma} 
\phi_w(\alpha \bmod U_{[w]}^{(1)})
=(-1)^{(|p|+1) \ord_{[w]}(\alpha)} \alpha^{(|p|+1)} \bmod U_{[w]}^{(1)}.
\end{equation}
In particular, we have
\[ \ord_{[w]}(\phi_w(\alpha \bmod U_{[w]}^{(1)}))
=(|p|+1) \ord_{[w]}(\alpha \bmod U_{[w]}^{(1)}).
\]
\end{lemma}
\begin{proof}
Note first that
$F(X)^\times_{[w]}/U_{[w]}^{(1)}$
is generated by the classes of 
non-zero constants $F^\times$ and
a single element $\alpha \in F(X)^\times$
such that $\ord_{[w]}(\alpha) = \pm 1$,
and hence it suffices to verify 
\eqref{eq:finallemma} for such elements.
We consider the map
\[ \Phi_w : F(X)^\times \overset{f^*}{\to}
 F(X')^\times \overset{g_*}{\to} F(X)^\times,
\]
which induces $\phi_w$ in view of
\[ F(X')_{[w]'} \times F(X')_{[w^*]'}
= F(X') \otimes_{F(X)} F(X)_{[w]}.
\]
Since $f$ and $g$ are of degree $|p|+1$,
one has $\Phi_w(\alpha)=\alpha^{|p|+1}$
for a constant $\alpha \in F^\times$,
showing \eqref{eq:finallemma} in this case.

As for the other type of generators,
we first consider the case $w=w_\infty=(-1, \dots, -1)$
(i.e. $[w]=[\infty]$).
Let us consider a diagram
\[
\xymatrix{
X=X(N)  \ar[d]_-\pi &
X'=X_0(pN) \ar[l]_-g \ar[d]_-{\pi'} \ar[r]^-{f} &
X=X_0(N) \ar[d]_-{\pi}
\\
X_0(1) 
& X_0(p) \ar[l]^-{g'} \ar[r]_-{f'}
& X_0(1) 
}
\]
where 
$f'$ and $g'$ are defined similarly as above,
and $\pi$ and $\pi'$ are
given by forgetting level structures (like $f$).
We may take $\alpha:=\pi^*(j)$ as
the pull-back of the $j$-function
$X_0(1) \overset{\cong}{\longrightarrow} \P^1$
so that $\ord_{[w_\infty]}(\alpha)=-1$.
Since the squares in the above diagram are Cartesian,
we have 
\[ \Phi_{w_\infty}(\alpha)
=g_* f^* \pi^*(j)=g_* {\pi'}^* {f'}^*(j)
=\pi^* {g'}_* {f'}^*(j).
\]
Therefore \eqref{eq:finallemma} 
is reduced to the case $N=1$, which we now assume.
Let $W_p' : X'=X_0(p) \to X'$ be the Atkin-Lehner involution 
with respect to $p$ so that
we have $g=f \circ W_p'$
and therefore
$\Phi_{w_\infty}(j)=g_* f^*(j)$
is given by the norm of $j':=W_p'(f^*(j)) \in F(X')^\times$ 
with respect to the field extension $F(X')/F(X)=F(j)$ along $f$.
Let $F_p(X, Y) \in A[X, Y]$ be the modular polynomial attached to $p$
so that we have
$F_p(f^*(j), j')=0$
and
$F_p(f^*(j), Y)$ is
an irreducible, monic and of degree $|p|+1$ in $\C_\infty(f^*(j))[Y]$.
(cf.\ \cite[Chapter 5, \S 2, Theorem 3]{Lang}, \cite[Theorem 2.4]{Bae}).
Therefore we have
\[
\Phi_{w_\infty}(j)=(-1)^{|p|+1}F_p(j, 0) \equiv
(-1)^{|p|+1} j^{|p|+1} \mod U_{[w_\infty]}^{(1)}.
\]
This proves \eqref{eq:finallemma} in this case.

Finally, 
let us go back to general $N$
and consider other $w \in \W$.
Then we may take $\alpha := W_w(j)$.
Using the Atkin-Lehner involution
$W_w' : X' \to X'$ acting on $X'$,
we have a commutative diagram
\[
\xymatrix{
X  \ar[d]_-{W_w} &
X' \ar[l]_-g \ar[r]^-{f} \ar[d]_-{W_w'} &
X \ar[d]_-{W_w}
\\
X   &
X' \ar[l]_-g \ar[r]^-{f} &
X.
}
\]
Hence \eqref{eq:finallemma} follows 
from the case $w=w_\infty$ above.
\end{proof}

\begin{proof}[Proof of Proposition \ref{prop:sc-eisen}]
The above lemma shows
everything except the second statement of (1),
which we now prove. 
Hence we suppose that we are in the case (NF) and $s=1, p=2$.
Then any element of $L^0$ is 
represented by
$(\alpha, \beta) \in F(X)_{[0]}^\times \oplus F_(X)_{[\infty]}^\times$
such that $\ord_{[0]}(\alpha)=-\ord_{[\infty]}(\beta)=:a$.
%(Here $[0]=[1_\W]$ and $[\infty]=[w_\infty]$.)
By the lemma, we see that
$\widetilde{T}_2$ sends it to the class of
$((-1)^a \alpha^3, (-1)^a \beta^3)$,
but the image of $(-1, -1)$ in $\tC$ is trivial
because it is in the image of the diagonal map.
This completes the proof.
\end{proof}

\begin{remark}
In the case (NF), we actually obtain that $\widetilde{J}(F)_\Tor$ is never Eisenstein when $N$ is odd and $s >2$.
To prove this, let $c \in \sC$ be the divisor class represented by
$$ [1]-[1/p_1] - [p_1/N] + [1/N] \in \sD_2.$$ 
By Corollary \ref{cor: ker}, we have $\delta(c) = 0$.
From the exact sequence \eqref{eq:delta}, there exists $\widetilde{c} \in \widetilde{J}(F)_\Tor$ whose image in $\widetilde{J}(F)_\Tor$ is $c$.
We shall show that $\widetilde{c}$ is not annihilated by $\widetilde{T}_2-3$.
Take $\xi \in L^0$ so that 
the surjective map $L^0 \twoheadrightarrow \sD_2$
sends $\xi$ to $[1]-[p_1] - [p_1/N] + [1/N]$.
Here we identify $\W$ with $\{m \in A_+ \mid m|N\}$ as in \eqref{eq:w-a-c}.
Let $\widetilde{c}' \in \widetilde{C} \subset \widetilde{J}(F)$ be the image of $\xi$ under the map $L^0 \twoheadrightarrow \widetilde{C}$. 
Then the canonical map from $\widetilde{J}(F)$ to $J(F)$ also sends $\widetilde{c}'$ to $c$.
From the exact sequence
$$
1 \rightarrow \sD_3 \otimes F^\times \rightarrow \widetilde{J}(F) \rightarrow J(F) \rightarrow 0,
$$
we may identify $\widetilde{c}-\widetilde{c}'$ with an element in $\sD_3 \otimes F^\times$, which is Eisenstein. However, Lemma \ref{finallemma} indicates that $\widetilde{c}'$ cannot be killed by $\widetilde{T}_2 - 3$ when $s>2$. 
(The action of $\widetilde{T}_2$ on
the components $[1], [1/p_1], [p_1/N], [1/N]$ 
is given by $x \mapsto -x^3$,
while on other components like $[1/p_2]$ it is given by $x \mapsto x^3$.)
Therefore neither is $\widetilde{c}$.
\end{remark}

\begin{remark}
Let $\T_{\widetilde{J}}$ (resp.\ $\mathcal{E}_{\widetilde{J}}$) be the image of $\T$ (resp.\ $\mathcal{E}$) in $\End(\widetilde{J})$. Then the Eisenstein property of $\sD_3 \otimes \G_m \subset \widetilde{J}$ implies that
$$\T/ \mathcal{E} \cong \T_{\widetilde{J}}/\mathcal{E}_{\widetilde{J}} \cong \Z,$$
unless $N=1$ (in which case we have $\widetilde{J}=J=0$).
Let $\T_J$ (resp.\ $\mathcal{E}_J$) be the image of $\T$ (resp.\ $\mathcal{E}$) in $\End(J)$. It is known that $\T_J/\mathcal{E}_J$ is always finite. 
This is a remarkable difference between $\T_{\widetilde{J}}$ and $\T_J$.
\end{remark}

\noindent
{\bf Acknowledgment.}
We thank the referee for careful reading
which helped improving our manuscript.

\bibliographystyle{plain}

\end{document}